\documentclass[12pt]{amsart}

\usepackage[all]{xy}
\usepackage{xypic,amsmath,amsfonts,amsthm,amssymb,amstext,amsbsy,graphicx,comment,epsfig}
\usepackage{structure}
\usepackage{comment}
\usepackage{mathtools}
\begin{document}
\selectfont
\define\lG{\lambda(G)}
\define\lg{\lambda(g)}
\define\Gal{Gal}
\define\lgi{\lambda(g)^{-1}}
\def\on#1{\begin{equation*}#1\end{equation*}}
\def\al#1{\begin{align*}#1\end{align*}}
\def\gen#1{\langle #1\rangle }
\define\rat{\mathbb{Q}}
\define\Fq{\mathbb{F}_q}
\define\Fqtoo{\mathbb{F}_{q^{p^2}}}
\define\Fqn{\mathbb{F}_{q^{{p^n}}}}

\newcommand{\verteq}{\rotatebox{90}{$\,=$}}
\newcommand{\equalto}[2]{\underset{\scriptstyle\overset{\mkern4mu\verteq}{#2}}{#1}}
\define\trip[#1][#2][#3]{\underset{#1}{\underset{#2}{#3}}}
\define\br[#1]{\{#1\}}

\def\keywords#1{\def\@keywords{#1}}
\parskip=0.125in
\title{Normality and Short Exact Sequences of Hopf-Galois Structures}
\author{Alan Koch}
\address{Department of Mathematics, Agnes Scott College, 141 E. College Ave., Decatur, GA 30030 USA}
\email{akoch@agnesscott.edu}
\author{Timothy Kohl}
\address{Department of Mathematics and Statistics, Boston University, 111 Cummington Mall, Boston, MA 02215 USA}
\email{tkohl@math.bu.edu} 
\author{Paul J.~Truman}
\address{School of Computing and Mathematics, Keele University, Staffordshire, ST5 5BG, UK}
\email{P.J.Truman@Keele.ac.uk}
\author{Robert Underwood}
\address{Department of Mathematics and Computer Science, Auburn University at Montgomery, Montgomery, AL 36124 USA}
\email{runderwo@aum.edu}
\maketitle
\begin{abstract}
Every Hopf-Galois structure on a finite Galois extension $K/k$ where $G=Gal(K/k)$ corresponds uniquely to a regular subgroup $N\leq B=\operatorname{Perm}(G)$, normalized by $\lambda(G)\leq B$, in accordance with a theorem of Greither and Pareigis. The resulting Hopf algebra which acts on $K/k$ is $H_N=(K[N])^{\lG}$. For a given such $N$ we consider the Hopf-Galois structure arising from a subgroup $P\triangleleft N$ that is also normalized by $\lG$. This subgroup gives rise to a Hopf sub-algebra $H_P\subseteq H_N$ with fixed field $F=K^{H_P}$. By the work of Chase and Sweedler, this yields a Hopf-Galois structure on the extension $K/F$ where the action arises by base changing $H_P$ to $F\otimes_k H_P$ which is an $F$-Hopf algebra.  We examine this analogy with classical Galois theory, and also examine how the Hopf-Galois structure on $K/F$ relates to that on $K/k$. We will also pay particular attention to how the Greither-Pareigis enumeration/construction of those $H_P$ acting on $K/F$ relates to that of the $H_N$ which act on $K/k$. In the process we also examine short exact sequences of the Hopf algebras which act, whose exactness is directly tied to the descent theoretic description of these algebras.
\end{abstract}
\noindent {\it key words:} Hopf-Galois extension, Greither-Pareigis theory, regular subgroup\par
\par
\renewcommand{\thefootnote}{}
\section{Introduction}
A separable extension of fields $K/k$ is Hopf-Galois if there exists a $k$-Hopf algebra $H$ (with $\Delta:H\rightarrow H\otimes_k H$ the co-multiplication, and $\epsilon:H\rightarrow k$ the co-unit) together with a $k$-algebra homomorphism $\mu:H\rightarrow End_k(K)$ such that 
\begin{align*}
\mu(h)(xy)&=\sum_{(h)}\mu(h_{(1)})(x)\mu(h_{(2)})(y),\\
\mu(h)(1)&=\varepsilon(h)(1),\\
\end{align*}
where $\Delta(h)=\sum_{(h)}h_{(1)}\otimes h_{(2)}$, the fixed ring 
$$K^H=\{x\in K| \mu(h)(x)=\varepsilon(h)x\ \forall h\in H\}$$
is precisely $k$, and the induced map $1\#\mu:K\# H\rightarrow End_k(K)$ is an isomorphism of $k$-algebras. The canonical example of a Hopf-Galois extension is the case where $K/k$ is Galois with $G=Gal(K/k)$. In this case, the group ring $k[G]$ acts as linear combinations of automorphisms which, by Dedekind's linear independence of characters, yields an action of $k[G]$ on $K/k$ which makes it Hopf-Galois. For extensions which are normal or not, the determination of which Hopf-Galois structures that may arise (if any) and how they act are described and enumerated using a result of Greither and Pareigis, the setup of which we outline here.\par
Let $\tilde K/k$ be the normal closure of $K/k$ where $G=Gal(\tilde K/k)$ and $G'=Gal(\tilde K/K)$. If $S=G/G'$ (left cosets) then $G$ acts on $S$ via $\lambda:G\hookrightarrow B=\operatorname{Perm}(S)$ where $\lambda(a)(bG')=abG'$. One has that $\tilde K\otimes K\cong \operatorname{Map}(S,\tilde K)$ where $\operatorname{Map}(S,\tilde K)$ is the $\tilde K$-algebra of set maps from $S$ to $\tilde K$ with basis $e_{s}$ for $s\in S$ which is also a $G$-set via $g(e_{s})=e_{\lambda(g)s}$. Moreover $\operatorname{Map}(S,\tilde K)/\tilde K$ is an $N$-Galois extension of rings precisely for any $N$ which is a regular (transitive and fixed-point free) subgroup of $B$. The action of an $n\in N$ on an element of $\operatorname{Map}(S,\tilde K)$ arises from letting $n(e_{s})=e_{n(s)}$. As such $\operatorname{Map}(S,\tilde K)/\tilde K$ is Hopf-Galois with respect to the action of $\tilde K[N]$. If $N$ is normalized by $\lG$ inside $B$ then one may descend this to get an action of $H=(\tilde K[N])^{\lG}$ on $\operatorname{Map}(S,\tilde K)^{\lG}/\tilde K^G$ where $(\tilde K\otimes K)^{\lG}\cong (\tilde K)^G\otimes K\cong k\otimes K\cong K$ and $\tilde K^G=k$ of course. That is $H$ yields a Hopf-Galois structure on $K/k$. Note,
$$
\operatorname{Map}(S,\tilde K)^{\lG}=\bigl\{\sum_{gG'\in S}g(x)e_{gG'}\ |\ x\in K\bigr\}
$$
which the reader can verify is a ring isomorphic to $K$ since the $e_{gG'}$ are orthogonal idempotents which form the basis of $\operatorname{Map}(S,\tilde K)$, together with the fact that $\tilde K^{G'}=K$. If $K/k$ is already normal then $G'$ is trivial, $\tilde K=K$, $S=G$, $\lG$ embeds in $\operatorname{Perm}(S)$ as the usual left regular representation, and 
\begin{equation}
K\cong\bigl\{\sum_{g\in G}g(x)e_{g}\ |\ x\in K\bigr\} 
\end{equation}
where $n(e_{g})=e_{n(g)}$ for $N\leq \operatorname{Perm}(G)$.\par
Whether $K/k$ is normal or not, we have the following.\par
{\bf Theorem: }\cite[Thm. 2.1]{GreitherPareigis1987} Let $K/k$ be a separable field extension, $S$ and $B$ as above, then there is a 1-1 correspondence between:\par
(a) Hopf-Galois structures on $K/k$\par
(b) regular subgroups $N\leq B$ that are normalized by $\lambda(G)\leq B$.\par
\noindent 
An important point to note is that the action of $H=(\tilde K[N])^{\lambda(G)}$ comes directly from the choice of regular subgroup, and not on the isomorphism class of $H$ as a $k$-Hopf algebra. Indeed, one can frequently find regular groups $N,M\leq B$, normalized by $\lambda(G)$, which are isomorphic as abstract groups, but (being distinct) give rise to distinct actions on $K/k$. Moreover, actions aside, the resulting fixed rings $H_N=(\tilde K[N])^{\lambda(G)}$ and $H_M=(\tilde K[M])^{\lambda(G)}$ may be isomorphic as Hopf algebras. The point is that the correspondence is between regular subgroups and pairs $(H,\mu)$ where $H$ is a $k$-Hopf algebra with a structure map $\mu:H\rightarrow End_k(K)$.\par
In the Greither-Pareigis framework, the `classical` action of $k[G]$ on $K/k$ when $G=Gal(K/k)$ corresponds to the regular subgroup $N=\rho(G)$ where $\rho:G\rightarrow B$ is the right regular representation, given by $\rho(g)(x)=xg^{-1}$. Since the left and right actions commute, then $H=(K[\rho(G)])^{\lG}=K^{G}[\rho(G)]\cong k[G]$. The point is that the action of $\lambda(G)$ in this situation is solely on the coefficients.\par
One of the interesting features of this theory is that a given extension $K/k$ may be Hopf-Galois with respect to variety of different Hopf algebras, including those arising from different (non-isomorphic) regular subgroups $N$. This is especially interesting in the case where $K/k$ is already classically Galois for then it is Hopf-Galois with respect to $k[G]$ and possibly for other forms of group rings $(K[N])^{\lG}$.\par
What we shall consider is the Hopf-Galois analogs of the usual facts from Galois theory about the correspondence between subgroups $J\leq G$ and intermediate fields $F=K^{J}$, as well as the correspondence between (normal) subgroups $J\triangleleft G$ and the set of (normal) intermediate fields. In particular we explore how the Greither-Pareigis theory applies in these situations. We review some known facts and then introduce new results. One of the key objectives is to understand the relationships between the actions of the various regular subgroups of the different ambient symmetric groups. We shall primarily consider the case where $K/k$ is already Galois with $G=Gal(K/k)$ and simultaneously Hopf-Galois with respect to $H_N=(K[N])^{\lambda(G)}$ for some regular subgroup $N\leq B$ where $N\neq\rho(G)$. We will also assume that $N$ has a normal subgroup $P$ which, under the right conditions, itself gives rise to a Hopf-Galois structure, much as the way a normal subgroup of a Galois group does classically. Moreover, since a normal subgroup $P\triangleleft N$ gives rise to the exact sequence of groups $1\rightarrow P\rightarrow N\rightarrow N/P\rightarrow 1$, we wish to consider the corresponding short exact sequence of Hopf algebras which arise, specifically those Hopf algebras which act on $K/k$, $K/F$, and $F/k$.
\section{Fixed Fields in Hopf-Galois Extensions}
Hopf-Galois extensions partially mirror some of the properties of ordinary Galois extensions. Assume $K/k$ is a (classically) Galois extension with group $G$. Every intermediate field $k\subseteq F\subseteq K$ corresponds bijectively to a subgroup $J=Gal(K/F)$ such that $K/F$ is Galois with group $J$; furthermore $F/k$ is Galois if and only if $J\triangleleft G$, and if so, $Gal(F/k)\cong G/J$. 
For Hopf-Galois extensions, the correspondence is given as follows (including the broader case where $K/k$ is not classically Galois), which is a synthesis of two related statements.\par
\begin{theorem}\cite[Thm. 7.6]{ChaseSweedler1969}  and \cite[Thm. 5.1]{GreitherPareigis1987}. If $K/k$ is Hopf-Galois with respect to a $k$-Hopf algebra $H$ then there is a map
$$
\operatorname{Fix}:\{k\text{ Hopf sub-algebras}\ H'\subseteq H\}\rightarrow \{\text{intermediate fields}\ k\subseteq F\subseteq K\}
$$
where 
$$F=\operatorname{Fix}(H')=K^{H'}=\{x\in K\ |\ \mu(h)(x)=\varepsilon(h)x\ \forall h\in H'\}$$
is the fixed field with respect to $H'$. This correspondence is injective and inclusion reversing. Moreover, $F\otimes H'$ acts on $K/F$ to make it Hopf-Galois.
\end{theorem}
For separable extensions, the Hopf algebras which arise are forms of group rings, which provides a fairly straightforward way of understanding what Hopf sub-algebras may appear. If $H'\subseteq H$ where $K\otimes H\cong K[N]$ then $K\otimes H'$ is a $K$-Hopf sub-algebra of $K[N]$. But being a Hopf sub-algebra of a group-ring one has, by for example \cite[Prop. 2.1]{Crespo-Separable}, or Exercise 2.1.25 in Radford \cite{Radford}, that $K\otimes H'$ must be of the form $K[P]$ for some $P\leq N$. However, in order that $H'=H_P=(K[P])^{\lG}$ be a sub-Hopf algebra of $H=H_N=(K[N])^{\lG}$ one must have that $P$ be normalized by $\lG$. This is summarized in the following, which paraphrases a number of statements in \cite{GreitherPareigis1987} adapted to the case considered here where we are assuming $K/k$ is classically Galois. (i.e. $\tilde K = K$) \par
\begin{proposition}\cite[Thm. 5.2 and Lemma p. 256]{GreitherPareigis1987} The Hopf sub-algebras of $(K[N])^{\lG}$ correspond to subgroups of $N$ that are stabilized by $p_g$ for all $g\in G$ where $\{p_g\ |\ g\in G\}$ is the cocycles of automorphisms coming from conjugation by each $g\in G$. That is, a given $P\leq N$ gives rise to a Hopf sub-algebra $H_P=(K[P])^{\lG}$ if and only if $\lG\leq Norm_B(P)$.
\end{proposition}
As such, if $P\leq N$ where $N$ is regular, and where both are normalized by $\lG$ then $H_N=(K[N])^{\lG}$ yields a Hopf-Galois structure on $K/k$, and similarly $H_P=(K[P])^{\lG}$ is such that $F\otimes H_P$ yields a Hopf Galois structure on $K/F$ where $F=K^{H_p}$. \par
Now, since $F$ is an intermediate field between $K$ and $k$, then $F=K^{J}$ the fixed field of some subgroup $J\leq G$ where $Gal(K/F)=J$. A natural question to ask is, since $K/F$ is Galois with group $J$ and acted on by $F\otimes H_P$ then how does it fit within the Greither-Pareigis framework? Since $F\otimes H_P$ acts on $K/F$ then it too should be a form of a group ring over $F$, namely $F[P]$. As observed by Crespo et al. in \cite[Prop. 8]{Crespo-Induced}, one has $F\otimes H_{P}\cong (K[P])^{\lambda(J)}$ with the implication that the embedding of $J$ in $\operatorname{Perm}(J)$ normalizes a regular subgroup of $\operatorname{Perm}(J)$ isomorphic to $P$ which is certainly very natural.\par
Recall that a regular subgroup of $B$ is one which is transitive and fixed-point free. As such, any subgroup of it will be fixed-point free, and such a subgroup is termed {\it semi-regular}. Note that regularity of $N\leq B$ can be defined as the property that $N$ acts fixed-point freely and that $|N|=|G|$. In this section we will explore the relationship between $P$ as a semi-regular subgroup of the regular subgroup $N\leq \operatorname{Perm}(G)$ and $P$ viewed as a regular subgroup of $\operatorname{Perm}(J)$. In particular, we shall look at how the action of $P\leq \operatorname{Perm}(G)$ on $J\leq G$ induces very naturally an embedding of $P\leq \operatorname{Perm}(J)$, as alluded to in \cite[Prop. 8]{Crespo-Induced} mentioned above. One of our tasks will be to relate these regular and semi-regular subgroups, with a view toward quantifying the action(s) that induce a Hopf-Galois structure on $F/k$ in the event that $P\triangleleft N$. Finally, the normality will also allow us to look at short exact sequences of Hopf algebras arising from actions on $K/k$, $K/F$ and $F/k$. \par

\subsection{From subgroups of $N$ to subgroups of $G$}\par
We begin with a question that has not yet been explored in the literature but which is in the background. How does $P$, giving rise to $F=K^{H_P}$, in turn give rise to that unique $J\leq G$ where $F=K^{J}$. We start with another variant of the $\operatorname{Fix}$ map, based on the equivalence
$$
\{\text{subfields}\ F\ | k\subseteq F\subseteq K\} \leftrightarrow \{\text{subgroups}\ J\leq G\}
$$
and examine the injective mapping from subgroups of $P\leq N$ normalized by $\lambda(G)$ to the set of subgroups $J\leq G$ due to the fact that $P$ gives rise to $F=K^{H_P}=K^{J}$ for exactly one $J\leq G$. 
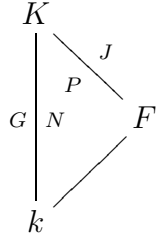
$$\diagram
& K  \ddline^{N}_{G} \drline_{P}^{J} \\
& &F \dlline  \\
&k
\enddiagram$$
The correspondence $P\mapsto J$ is a bit mysterious, but we can show (combinatorially) how $J$ arises from $P$. Also, since we wish to explore the possibility of putting a Hopf-Galois structure on $F/k$, then, following the analogy with ordinary Galois theory, we will ultimately restrict our attention to those $P$ that are normal in $N$. This will be necessary in order to construct a form of the group ring $k[N/P]$ that will act on $F/k$. However, initially we shall look at how $J$ arises from those $P$ normalized by $\lambda(G)$.\par 
As mentioned $K/k$ is a Galois extension with $G=Gal(K/k)$ which is also Hopf-Galois with respect to the action of $H_N=(K[N])^{\lG}$ where $N\leq B=\operatorname{Perm}(G)$ is a regular subgroup, that is normalized by $\lG$. Since $F$ is an intermediate field between $K$ and $k$ we have $F=K^{J}$ for $J$ a subgroup of $G$. To determine $J$ from $P$ we shall consider, in some depth, the action of $H_N$ and $H_P$ on the copy of $K$ embedded in $\operatorname{Map}(G,K)$, namely
$$K\cong\{\sum_{g\in G}g(x)e_{g}\ |\ x\in K\}$$ 
mentioned above.\par
For $h\in H_{N}$ one has $h=\sum_{n\in N}c_n\cdot n$ where $c_n\in K$ and $n\in N$ where $\gamma(c_n)=c_{\lambda(\gamma)n\lambda(\gamma)^{-1}}$ for all $\gamma\in G$, since $h$ is fixed by all of $\lG$. Since $N$ acts on $G$, it is useful to consider certain 'slices' of $N\times G$ that arise from this action. We define
$$
Z_{N,t}=\{(n,g)\in N\times G\ |\ n(g)=t\}
$$
for each $t\in G$. The element $h$ of $H_N$ acts on a typical element of $K$ as follows 
\begin{align*}
h(\sum_{g\in G}g(x)e_g)&=\sum_{n\in N}c_n\left(\sum_{g\in G}g(x)e_{n(g)}\right)\\
                      &=\sum_{g\in G}\sum_{n\in N}c_ng(x)e_{n(g)}\\
                      &=\sum_{t\in G}\left(\sum_{(n,g)\in Z_{N,t}} c_ng(x)\right)e_t\\
                      &=\sum_{g\in G}g(y)e_{g}
\end{align*}
for some $y\in K$. This implies that 
$$
t(y)=\left(\sum_{(n,g)\in Z_{N,t}} c_ng(x)\right)
$$
for all $t\in G$, in particular for $t=i_G$ (the identity) which means 
\begin{align*}
y&=\sum_{(n,g)\in Z_{n,i_G}} c_ng(x)\\
 &=\sum_{n\in N} c_n n^{-1}(i_G)(x)\\
\end{align*}
by regularity of the action of $N$ on $G$.
Now, since $K/k$ is Hopf-Galois with respect to $H_N$ then $\operatorname{Fix}(H_N)=K^{H_N}$ which are those elements $x\in K$ such that 
$$h(x)=\varepsilon(h)x=(\sum_{n\in N}c_n)x$$
which, since $\{ n^{-1}(i_G)\ |\ n\in N\}=G$, is precisely the field $k$ since $K^{G}=k$. Similarly, for $H_P$ we have $F=K^{H_P}$. Now, $H_P\subseteq H_N$ in a natural way since $P\leq N$, in particular those $h\in H_N$ of the form $\sum_{q\in P}d_q\cdot q$ where $\gamma(d_q)=d_{\lambda(\gamma)q\lambda(\gamma)^{-1}}$ for all $\gamma\in G$. For these elements of $H_P$ we have, for $x\in K$, that
\begin{align*}
h(x)=h(\sum_{g\in G}g(x)e_g)&=\sum_{q\in P}d_q\left(\sum_{g\in G}g(x)e_{q(g)}\right)\\
                      &=\sum_{t\in G}\left(\sum_{(q,g)\in Z_{P,t}} d_q(x)\right)e_t\\
                      &=\sum_{q\in P} d_q q^{-1}(i_G)(x)\\
\end{align*}
and since $P$ is a subgroup of the regular permutation group $N$, it is semi-regular. As such, $|\{q^{-1}(i_G)\ |\ q\in P\}|=|P|$. Also $x\in K^{H_P}$ if 
$$
h(x)=\sum_{q\in P} d_q q^{-1}(i_G)(x)=\varepsilon(h)x=(\sum_{q\in P} d_q)x
$$
and since $F=K^{H_P}$ is an intermediate field, then $F=K^{J}$ for some $J\leq G$ where $|J|=dim_k(H_P)=|P|$. We can relate $J$ to $P$ by considering the action of an element which is guaranteed to lie in $H_P$, namely $h=\sum_{q\in P}1\cdot q$. For $x\in F$ one has $h(x)=|P|x$ which implies that 
\begin{align*}
\sum_{q\in P} q^{-1}(i_G)(x)&=|P|x\\
                          &=|J|x\\
\end{align*}
where $\{q^{-1}(i_G)\ |\ q\in P\}$ is some subset of $G$. What we wish to show is that this subset is precisely $J$. To show this we use the following (somewhat more general statement) which seems standard, but no reference could be found.
\begin{proposition}
\label{fixedsum}
For $G=Gal(K/k)$ with $F=K^{J}$ for $J\leq G$ as above, if $S=\{\sigma_1,\dots,\sigma_r\}\subseteq G$ such that 
$$
\sum_{i=1}^{r} \sigma_i(x) = rx
$$
for all $x\in F$, then $S\subseteq J$.
\end{proposition}
\begin{proof}
The key to the proof is to show that if the sum of the elements of $S$ acts like a scalar multiple of the identity on $F$, then all $\sigma_i\in S$ are actually elements of $J$. So let's assume that for each $x\in F$ one has $\sigma_i(x)=x+\delta_{x,i}$. Since $\sum_{i=1}^{r}\sigma_i(x)=rx$ then we must have $\sum_{i=1}^r \delta_{x,i}=0$. If we use the fact that each $\sigma_i$ is an automorphism of $K$ (in particular that it is multiplicative), and if we assume $\sum_{i=1}^{r} \delta_{x,i}^l=0$ for $j=1\dots n-1$ for each $n$ then 
\begin{align*}
						 rx^n &=\sum_{i=1}^r \sigma_i(x^n)\\
						 &=\sum_{i=1}^r (x+\delta_{x,i})^n\\
                         &=\sum_{i=1}^r \sum_{j=0}^{n}\binom{n}{j} x^j\delta_{x,i}^{n-j}\\
                         &=\sum_{j=0}^{n}\sum_{i=1}^r \binom{n}{j} x^j\delta_{x,i}^{n-j}\\
                         &=\sum_{j=0}^{n} \binom{n}{j}x^j\sum_{i=1}^r \delta_{x,i}^{n-j}\\
                         &=\sum_{i=1}^r \delta_{x,i}^n+rx^n\\
\end{align*}
which means $\sum_{i=1}^r \delta_{x,i}^n=0$. To complete the proof we invoke the Newton-Girard formul\ae\  which, stated formally as \cite[Thm. 8 - Ch. 7]{Cox} are that (over a field containing $\mathbb{Q}$) the ring of symmetric functions in $r$ variables $t_1,\dots,t_r$ is generated by the power sum polynomials $\sum_{i=1}^r t_i^n$. Thus, for any $x\in F$ if we let $f(z)=(z-\epsilon_{x,1})\cdots (z-\delta_{x,r})$ then the coefficients of every term of $f(z)$ of degree lower than $r$ must be zero. That is, $f(z)=z^r$ so all its roots, namely the $\delta_{x,i}$, are identically zero. As such, $\sigma_i(x)=x$ for all $x\in F$ for $i=1,\dots, r$ so that each $\sigma_i\in J$.
\end{proof}
We use this now to show how $J$ corresponds to $P$.
\begin{theorem}
\label{commonorbit}
For $K/k$ with $G=Gal(K/k)$ and $N$, $P$ and $F=K^{H_P}$ as above, $\{q^{-1}(i_G)\ |\ q\in P\}=J$.
\end{theorem}
\begin{proof}
As seen above, for $x\in F$ one has $\sum_{q\in P} q^{-1}(i_G)(x)=|J|x$ where  $\{q^{-1}(i_G)\ |\ q\in P\}\subseteq G$. By Prop. \ref{fixedsum} above this set must consist of elements of $J$ and since $|P|=|J|$ this set is exactly $J$.
\end{proof}
We will need the following in the next section.
\begin{corollary}
\label{orbitcoset}
Given $N$, $P$, $G$ and $J$ as above, for $g\in G$,  $Orb_{P}(g)=\{q(g)\ |\ q\in P\}=gJ$, the left coset of $g$ in $J$.
\end{corollary}
\begin{proof}
Since $G$ normalizes $P$ then $\lambda(g)^{-1}P\lambda(g)=P$ and by Thm. \ref{commonorbit} above, $Orb_P(i_G)=J$ which implies
$$J=\{\lambda(g)^{-1}q\lambda(g)(i_G)\ |\ q\in P\}=\{g^{-1}q(g)\ |\ q\in P\}$$
and so left multiplying both sides by $g$ yields the result.
\end{proof}
Now, given that sub-Hopf algebras of $H_N$ correspond injectively to intermediate fields $k\subseteq F\subseteq K$, which, in turn, corresponds bijectively to subgroups of $J\leq G$, we also have
\begin{corollary}
\label{PtoJ}
The correspondence 
$$
\{\text{subgroups}\ of\ N\ \text{normalized by} \lG\}\overset{\Psi}{\longrightarrow}\{\text{subgroups of}\ G\}
$$
given by $\Psi(P)=Orb_{P}(i_G)$ is injective where the subset $Orb_{\lambda(J)}(i_G)$ is the sub{\it group} corresponding to the Galois group of $K/K^{H_P}$.
\end{corollary}
This is intriguing in that we know that $P$ {\it must} correspond to $J=Gal(K/K^{H_P})$, but the precise relationship (as subgroups of $\operatorname{Perm}(G)$) between $P$ and $\lambda(J)$ is still a bit mysterious. Of course, $\Psi({e_N})=\{i_G\}$ and $\Psi(N)=G$ but when looking at different specific examples by hand or using GAP \cite{GAP4}, it does not seem that one can recover $J$ in any kind of obvious way from $P$ as say a normalizer or centralizer of something in $\operatorname{Perm}(G)$. Rather the correspondence here is seemingly based on purely combinatorial considerations, namely that the elements of $P$ must give rise to a bijection (at the level of sets) from $J$ to itself within $\operatorname{Perm}(G)$. This will be of key importance in examining the relationship between the actions of $\lambda(G)$ and $\lambda(J)$ as compared to the actions of $N$ and $P$, in particular the role that $P$ being normal in $N$ plays, as well as the differences that arise when $J\triangleleft G$ versus when it isn't.\par  
At this stage, there are two natural questions to ask. Is $\Psi$ onto, and for each such $P\leq N$, is $\Psi(P)\cong P$?
From computations done using GAP in degree $42$, based on the second author's work in \cite{Kohl2013}, the answer to both question is {\it no}. The 6 groups of order 42 are $C_{42}$, $C_7\times D_3$, $(C_7\rtimes C_3)\times C_2$, $C_3\times D_7$, $D_{21}$, and $(C_7\rtimes C_3)\rtimes C_2$. One of the reasons for looking at this class of examples is that they are well understood, and also, since $42$ is the product of three primes, there is a lot more variety in the number of isomorphism classes of subgroups. The distribution of $N$ for each $G$ (where the column entries are $G$ and the rows are the number of $N$ of the given isomorphism class) are summarized below, where we show (in $\{\}$) how many of the given $N$ give rise to $\Psi$ which are onto.\par\vskip0.125in
\centerline{\begin{tabular}{|c|c|c|c|c|c|c|}\hline
$G\downarrow\setminus N\rightarrow$ & $C_{42}$ & $C_7\times D_3$ & $C_2\times (C_7\rtimes C_3)$ & $C_3\times D_7$& $D_{21}$ & $(C_7\rtimes C_3)\rtimes C_2$\\ \hline
$C_{42}$ & $1\ \{1\}$ & $2\ \{1\}$ & $4\ \{2\}$ & $2\ \{1\}$ & $4\ \{1\}$ & $4\ \{2\}$\\ \hline
$C_7\times D_3$& $3\ \{0\}$ & $2 \{1\}$ & $0$ & $6\ \{0\}$ & $4\ \{1\}$ & $0$\\ \hline
$C_2\times (C_7\rtimes C_3)$& $7\ \{0\}$ & $14\ \{0\}$ & $16\ \{1\}$ & $14\ \{0\}$ & $28\ \{0\}$ & $28\ \{0\}$\\ \hline
$C_3\times D_7$& $7\ \{0\}$ & $14\ \{0\}$ & $28\ \{0\}$ & $2\ \{1\}$ & $4\ \{1\}$ & $28\ \{0\}$\\ \hline
$D_{21}$& $21\ \{0\}$ & $14\ \{0\}$ & $0$ & $6\ \{0\}$ & $4\ \{1\}$ & $0$\\ \hline
$(C_7\rtimes C_3)\rtimes C_2$& $7\ \{0\}$ & $14\ \{0\}$ & $28\ \{0\}$ & $14\ \{0\}$ & $28\ \{0\}$ & $16\ \{1\}$\\ \hline
\end{tabular}}\par
For $N=\rho(G)$ we have that every subgroup of $N$ is normalized by $\lambda(G)$ since $\rho(G)$ and $\lambda(G)$ centralize each other. In this case $\Psi$ is trivially onto, and the above table shows that frequently $\Psi$ is onto {\it only} for $N=\rho(G)$. Other than the fact that $|P|=|\Psi(P)|$, the groups $P$ and $J$ can be in different isomorphism classes. For example, when $G\cong C_{42}$ there is an $N\cong C_3\times D_7$ with $P\leq N$ where $P\cong C_{14}$ but $\Psi(P)\cong D_7$. In the final section we shall summarize more information about the degree 42 cases.\par
\subsection{Semi-regular subgroups and orbits}\par
As mentioned above, $P\leq N\leq \operatorname{Perm}(G)$, normalized by $\lambda(G)\leq B$ gives rise to $J=\Psi(P)$ where, since $N$ is regular, we have that $P$ is semi-regular but not transitive. Moreover, we also have that $\lambda(J)\leq\lambda(G)\leq \operatorname{Perm}(G)$ also normalizes $P$. And, as also mentioned earlier, $F\otimes H_P\cong (K[P])^{\lambda(J)}$ due to the fact $K/F$ is Hopf-Galois with respect to the action of a copy of $P$ embedded as a regular subgroup of $\operatorname{Perm}(J)$ normalized by $\lambda(J)\leq \operatorname{Perm}(J)$. These two views are tied together by Thm. \ref{commonorbit} which implies that $P\leq \operatorname{Perm}(G)$ restricted to $J$ yields this aforementioned regular subgroup of $\operatorname{Perm}(J)$. The virtue of this is that it mirrors the fact that $\lambda(G)$ restricted to $J$ is naturally equal to $\lambda(J)\leq \operatorname{Perm}(J)$. In order to construct a Hopf-Galois structure on $F/k$ we shall assume from this point onward that $P\triangleleft N$. What this yields for us is a way to view $\lambda(G)$, $\lambda(G)/\lambda(J)$ (if $J\triangleleft G$), and $N/P$ as subgroups of a fixed $\operatorname{Perm}(S)$ for a specifically chosen set $S$, where $N/P$ acts regularly and $\lambda(G)/\lambda(J)$ normalizes this regular subgroup. This will lead naturally to a Hopf-Galois structure on $F/k$ with associated group $N/P$. If $J$ is not a normal subgroup then we can still construct a Hopf-Galois structure on $F/k$ corresponding to $N/P$ but the argument will be somewhat different.\par
We start by establishing some additional notation, namely 
\begin{itemize}
\item $[K:F]=|J|=|P|=p$
\item $[F:k]=[G:J]=[N:P]=m$ 
\item $[K:k]=|N|=|G|=mp=n$
\end{itemize}
where $p$ is not necessarily a prime. Our first step is to find this common set on which $N/P$ and $G/J$ both act. Again, we shall initially assume that $J\triangleleft G$ and later on relax this condition. However, we will choose the sets/actions so that in the non-normal case, we won't have to modify the arguments that much. With Cor. \ref{orbitcoset} in mind, together with the fact that $G$ acts naturally on the left cosets of $J$ in $G$, even if $J$ is not normal in $G$, we have 
$$
G=g_1J\cup g_2J\cup\dots\cup g_mJ,
$$
where $\{g_1,\dots,g_m\}$ are a transversal of $J$ in $G$, where we may assume $g_1=i_G$. We start with the following basic fact.
\begin{lemma}
\label{regblock}
For $G$, $J$, $N$ and $P$ as above, the quotient groups $\lambda(G)/\lambda(J)$ and $N/P$ act on the cosets $\{g_1J,\dots,g_mJ\}$ as regular permutation groups.
\end{lemma}
\begin{proof}
We start by observing that $\lambda(G)$ acts transitively on $\{g_iJ\}$ due to $\lambda(G)$ being regular, hence transitive on $G$ itself. Similarly,  since $J\triangleleft G$ then $\lambda(G)/\lambda(J)$ also acts transitively on $\{g_iJ\}$ since given $g_iJ$ and $g_jJ$, then for $h\in J$, 
\begin{align*}
\lambda(g_jg_i^{-1}h)(g_iJ)&=g_jg_i^{-1}hg_iJ\\
                          &=g_jg_i^{-1}g_ih'J\text{ for some $h'\in J$ by normality}\\
                          &=g_jJ\\
\end{align*}
and since $|G/J|=m$ then this transitive action of $\lambda(G)/\lambda(J)$ is regular. If $n\in N$ then by Cor. \ref{orbitcoset}
\begin{align*}
n(gJ)&=\{n(q(g))\ |\ q\in P\}\\
     &=\{q(n(g))\ |\ q\in P\}\text{ since $P\triangleleft G$}\\
     &=n(g)J\\
\end{align*}
which means $N$ operates transitively on the cosets, so that $N/P$ acts transitively and therefore regularly since $[N:P]=[G:J]$.
\end{proof}
Before proceeding further it is worth noting how the normality of the given subgroups $J$ and $P$ strongly impacts the choice of object on which both $\lambda(G)/\lambda(J)$ and $N/P$ act. By itself, $\lambda(J)$ gives rise to orbits which are precisely the {\it right} cosets of $J$ in $G$, and, as seen above, the orbits of $P$ on $G$ are precisely the left cosets of $J$ in $G$. When $J$ is normal in $G$ these coincide of course, but if $J$ is not normal in $G$ then $G$ would still act naturally on its left cosets, i.e. the orbits of $P$. Note also, in the second half of the above proof that the transitivity of the action of $N$ on the orbits of $P$ uses the normality of $P$ in $N$. This is partly why we require that $P\triangleleft N$ but allow for the possibility of $J$ not being normal in $G$. The fact that $N$ and $P$ are normalized by $\lambda(G)$ yields the following.
\begin{theorem}
\label{QnormQ}
For $\lambda(J)\triangleleft\lambda(G)$ and $P\triangleleft N$ as above, $\lambda(G)/\lambda(J)$ normalizes $N/P$ as a subgroup of $\operatorname{Perm}(\{g_iJ\})$.
\end{theorem}
\begin{proof}
Let $n\in N$, $q\in P$, $\gamma\in G$, $h_1,h_2\in J$ and consider 
$\lambda(\gamma h_1)n q\lambda(\gamma^{-1}h_2)\in (\lambda(\gamma J)(nP)(\lambda(\gamma^{-1} J))$ and observe that
\begin{align*}
\lambda(\gamma h_1)n q\lambda(\gamma^{-1}h_2)[g_iJ]&=\lambda(\gamma h_1)n q [\gamma^{-1}h_2g_iJ]\\
                                                          &=\lambda(\gamma h_1)n q[\gamma^{-1}g_iJ]\text{ since $J\triangleleft G$} \\
                                                          &=\lambda(\gamma h_1)[n(\gamma^{-1}g_i)J] \text{ since $P$ acts trivially on $\{g_iJ\}$}\\
                                                          &=\lambda(\gamma) h_1[n (\gamma^{-1} g_i)J)] \\
                                                          &=\lambda(\gamma) [n(\gamma^{-1} g_i)J]\text{ again since $J\triangleleft G$}  \\
                                                          &=[\gamma n(\gamma^{-1} g_i)J].\\
\end{align*}
The last set above is therefore $\lambda(\gamma)n\lambda(\gamma^{-1})P[g_iJ]$. That is 
$$\lambda(\gamma J)(nP)\lambda(\gamma^{-1}J)=\lambda(\gamma)n\lambda(\gamma^{-1})P\in N/P$$ 
describes the action of $\lambda(G)/\lambda(J)$ on $N/P$ in $\operatorname{Perm}(\{g_iJ\})$.
\end{proof}
\noindent We can now prove the following.
\begin{theorem}
\label{mainFknormal}
The extension $F/k$ is Hopf-Galois with respect to the Hopf algebra $H_{N/P}=(F[N/P])^{\lambda(G)/\lambda(J)}$.
\end{theorem}
\begin{proof}
By Lemma \ref{regblock} $\lambda(G)/\lambda(J)$ and $N/P$ act on $\{g_iJ\}$ (which are the cosets of $J$ in $G$) where $N/P$ is a regular subgroup of $\operatorname{Perm}(\{g_iJ\})$. This implies that $\operatorname{Map}(G/J,F)/F$ is Galois with respect to $N/P$ and therefore that 
$$(\operatorname{Map}(G/J,F))^{G/J}/F^{G/J}\leftrightarrow F/k$$
is Galois with respect to $N/P$ as well. As such by Thm. \ref{QnormQ}, $H_{N/P}$ acts on $F/k$ to make it Hopf-Galois 
\end{proof}
\subsection{$J$ not normal in $G$}\par
What happens when $J$ is not normal in $G$? Even though we assume $P\triangleleft N$, there is no reason to expect that $J=\Psi(P)$ given in Cor. \ref{PtoJ} is normal in $G$. Using GAP we generated the following example in degree 24:
\begin{align*}
G&=\langle ( 1, 2)( 3,13)( 4, 8)( 5, 7)( 6, 9)(10,21)(11,20)(12,16)(14,18)(15,17)(19,24)(22,23),\\
&( 1, 3, 9)( 2, 6,13)( 4,11,23)( 5,19,17)( 7,15,24)( 8,22,20)(10,18,12)(14,21,16),\\
&( 1, 4)( 2, 7)( 3,10)( 5,12)( 6,14)( 8,16)( 9,17)(11,19)(13,20)(15,22)(18,23)(21,24)\\
&( 1, 5)( 2, 8)( 3,11)( 4,12)( 6,15)( 7,16)( 9,18)(10,19)(13,21)(14,22)(17,23)(20,24)\rangle\\
&\cong S_4. \\
\text{ Then } N&=\langle ( 1, 6, 9, 2, 3,13)( 4,15,23, 7,11,24)( 5,22,17, 8,19,20)(10,21,12,14,18,16),\\
&( 1, 9, 3)( 2,13, 6)( 4,23,11)( 5,17,19)( 7,24,15)( 8,20,22)(10,12,18)(14,16,21),\\
&( 1, 4)( 2, 7)( 3,10)( 5,12)( 6,14)( 8,16)( 9,17)(11,19)(13,20)(15,22)(18,23)(21,24),\\
&( 1,12)( 2,16)( 3,19)( 4, 5)( 6,22)( 7, 8)( 9,23)(10,11)(13,24)(14,15)(17,18)(20,21)\rangle\\
&\cong A_4\times C_2\\
\end{align*}
is regular and normalized by $G$, and $N$ has a normal subgroup $P$\par\vskip0.125in
\begin{align*}
P&=\langle ( 1, 4)( 2, 7)( 3,10)( 5,12)( 6,14)( 8,16)( 9,17)(11,19)(13,20)(15,22)(18,23)(21,24),\\
&( 1, 5)( 2, 8)( 3,11)( 4,12)( 6,15)( 7,16)( 9,18)(10,19)(13,21)(14,22)(17,23)(20,24),\\
&( 1, 2)( 3, 6)( 4, 7)( 5, 8)( 9,13)(10,14)(11,15)(12,16)(17,20)(18,21)(19,22)(23,24)\rangle\\
&\cong C_2\times C_2 \times C_2\\
\end{align*}
which is also normalized by $G$, but $G$ has no normal subgroups of order 8. Now, working with this $G\leq S_{24}$ as opposed to $\lambda(G)\leq \operatorname{Perm}(G)$ it is not so easy to directly apply Cor. \ref{PtoJ} to get $J=Orb_{P}(i_G)$. However, one can show that $G$ above (being isomorphic to $S_4$) contains only three subgroups of order 8, all isomorphic to $D_4$. If we make the identification $1\leftrightarrow i_G$ then we can infer which is {\it the} $J=\Psi(P)$. To see this, one first observes that $P$ gives rise to the following orbits 
\begin{align*}
&\{ 1, 2, 4, 5, 7, 8, 12, 16\}\\
&\{ 3, 6, 10, 11, 14, 15, 19, 22 \}\\
&\{ 9, 13, 17, 18, 20, 21, 23, 24 \}\\
\end{align*}
and for the three subgroups of order 8, $J_1$, $J_2$ and $J_3$ we have orbits
 \begin{align*}
J_1 &\rightarrow  \{ 1,2,4,5,7,8,12,16\},\{ 3,10,11,13,19,20,21,24 \},\{6,9,14,15,17,18,22,23\}\\
J_2 &\rightarrow  \{ 1,4,5,6,12,14,15,22\}, \{ 2,3,7,8,23,10,11,16,19\}, \{9,13,17,18,20,21,24\}\\
J_3 &\rightarrow  \{ 1,4,5,12,13,20,21,24\}, \{2,7,8,9,16,17,18,23\}, \{3,6,10,11,14,15,19,22\}\\
\end{align*}
which, again identifying $1$ with $i_G$ implies that $J_1=\Psi(P)$ since the orbit containing 1 that matches $[ 1,2,4,5,7,8,12,16]$ is that belonging to $J_1$ above. What one also notices is that the other two orbits 
$$\{ 3,10,11,13,19,20,21,24 \}\text{ and }\{6,9,14,15,17,18,22,23\}$$ 
do not match 
$$\{ 3, 6, 10, 11, 14, 15, 19, 22 \}\text{ and }\{ 9, 13, 17, 18, 20, 21, 23, 24 \}$$
above arising from $P$, which is, of course, due to the fact that the orbits of $J$ (resp. $P$) are the right (resp. left) cosets of $J$ in $G$.\par 
The case where $J$ is non normal in $G$ synchronizes naturally with the initial setup of Greither-Pareigis as mentioned earlier.
They start with a non-Galois extension and then look to the normal closure of this extension in order to construct a Hopf-Galois structure if possible on the original extension. In the case where $J$ is not normal in $G$, the extension $F/k$ is not classically Galois of course, but we will show that it is Hopf-Galois, which begins by looking to its normal closure. The normal closure $\tilde F/k$ has as its Galois group $G/I$ where $I=\cap_{g\in G}gJg^{-1}$, and clearly $\tilde F\subseteq K$ as diagrammed below. 
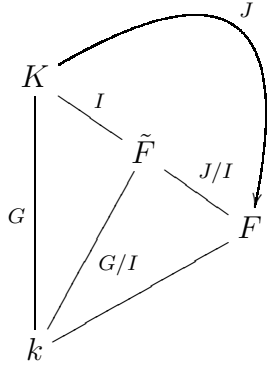
\begin{displaymath}
    \xymatrix@R=1em{
     K\ar@{-}[dr]^{I} \ar@{-}[dddd]_{G}\ar@/^5pc/[ddrr]^{J}            &        \\ 
                      & \tilde F \ar@{-}^{G/I}[dddl]\ar@{-}^{J/I}[dr]&        \\
                      &                              &   F  \ar@{-}[ddll]  \\
        &                              &        \\
     k                &                              &        \\    
}
\end{displaymath}
Looking back at the basic setup in \cite{GreitherPareigis1987}, principally \cite[Lemma 1.1 and Lemma 1.2]{GreitherPareigis1987} we have that $S$ (the quotient of the respective Galois groups) is going to be $(G/I)/(J/I)$. This quotient of quotients looks formidable, but there is a natural identification of it with $G/J$ where, of course, $|G/J|=|N/P|=[F:k]$. Following this development further leads to the question of finding a regular subgroup  $Q\leq \operatorname{Perm}(S)$ which leads to an action making $\operatorname{Map}(S,\tilde F)/\tilde F$ Hopf-Galois with respect to $\tilde F[Q]$. If $Q$ is normalized by $\lambda(G)/\lambda(I)$ then $\operatorname{Map}(S,\tilde F)^{\lambda(G)/\lambda(I)}/\tilde F^{G/I}$ (i.e. $F/k$) is Hopf-Galois with respect to $(\tilde F[Q])^{G/I}$. However, we can replace $\tilde F$ with any field which contains it, most naturally $K$ since $Gal(K/k)$ maps onto $Gal(\tilde F/k)$. Indeed, if $I$ is trivial then $\tilde F=K$. Therefore we look for regular $Q\leq \operatorname{Perm}(S)$ normalized by $\lambda(G)$ giving rise to $H_Q=(K[Q])^{\lambda(G)}$. We may thereby dispense with the need to worry about $I$ at all since (as indicated earlier) $S$ can be identified with the (set) $G/J$.\par
By Lemma \ref{regblock} we have that $N/P$ acts regularly, and $G$ acts (transitively) on $S=\{g_iJ\}$. The only question is whether the analog of Thm. \ref{QnormQ} holds for the action of $G$ on $N/P$ within $\operatorname{Perm}(S)$. The answer is yes by simply setting $h_1=i_G$ and $h_2=i_G$ in $\lambda(\gamma h_1)$ and $\lambda(\gamma^{-1}h_2)$ in the proof. This yields the action $\lambda(\gamma)(\eta P)\lambda(\gamma)^{-1}=\lambda(\gamma)\eta\lambda(\gamma)^{-1}P$, so that, indeed, $\lambda(G)$ normalizes $N/P$ in $\operatorname{Perm}(S)$ which yields the following variation on Thm. \ref{mainFknormal}.
\begin{theorem}
\label{mainFknon-normal}
The extension $F/k$ is Hopf-Galois with respect to the Hopf algebra $H_{N/P}=(K[N/P])^{\lambda(G)}$ where $\lambda(G)\leq \operatorname{Perm}(S)$.
\end{theorem}

\section{Short Exact Sequences}
As seen above, the analysis of how to construct Hopf-Galois structures on intermediate fields is dependent on when the relevant subgroups, $P$ and $J=\Psi(P)$ are normal or not. One basic consequence of normality at the level of groups is the existence of an exact sequence. Such an exact sequence gives rise to a sequence of group rings, which are Hopf algebras. We consider these sequences, as well as those corresponding to the fixed rings which are the Hopf-algebras which act on the different intermediate fields.\par
In ordinary Galois theory, if $K/k$ is Galois with $J\triangleleft G:=Gal(K/k)$ then $K^{J}/k$ is Galois with group $G/J$. One has an exact sequence of groups 
$$
1\rightarrow J\rightarrow G\rightarrow G/J\rightarrow 1
$$
where, since $J\triangleleft G$, then by \cite[Prop 4.14]{Ch00}, the induced Hopf-algebra maps $k[J]\rightarrow k[G]$ (resp. $k[G]\rightarrow k[G/J]$) are normal (resp. co-normal). As such, 
$$
k\rightarrow k[J]\rightarrow k[G]\rightarrow k[G/J]\rightarrow k
$$
is a short exact sequence of $k$-Hopf algebras.\par\vskip0.125in
But in terms of the actions on the relevant (intermediate) fields, these are not exactly the Hopf algebras that act to make the given field extensions Hopf-Galois. If we look at the Hopf-Galois actions induced by the Galois groups $J$, $G$, and $G/J$ then the Hopf algebras are group rings
\begin{align*}
K/K^{J}\text{ is acted on by }& (K[\rho(J)])^{\lambda(J)}\cong K^{J}[J]\\
\text{ where }&\rho(J),\lambda(J)\leq \operatorname{Perm}(J)\\
&\\
K/k\text{ is acted on by }& (K[\rho(G)])^{\lambda(G)}\cong k[G]\\
&\\
K^{J}/k\text{ is acted on by }&(K^{J}[\rho(G/J)])^{\lambda(G/J)}\cong k[G/J]\\
\text{ where }&\lambda(G/J),\rho(G/J)\leq \operatorname{Perm}(G/J)\\
\end{align*}
The latter two Hopf algebras are defined over $k$ but the first is not, which is not unexpected since it is acting with respect to the ground field $K^{J}$.\par\vskip0.125in
Since $\rho(J)$ is normalized by $\rho(G)$ and $\lambda(G)$ then $(K[\rho(J)])^{\lambda(G)}$ is a Hopf sub-algebra of $(K[\rho(G)])^{\lambda(G)}$ by \cite[Thm. 7.6]{ChaseSweedler1969}, and $$(K[\rho(J)])^{\lambda(J)}\cong K^{J}\otimes (K[\rho(J)])^{\lambda(G)}$$.\par
Part of the simplicity of the classical case above is that the Hopf-Algebras are group rings. For $K/k$ Hopf-Galois under the action of $H_N=(K[N])^{\lG}$, where $N$ is a regular subgroup of $\operatorname{Perm}(G)$ normalized by $\lambda(G)$, we would like to consider $P\triangleleft N$, also normalized by $\lambda(G)$ and the Hopf-Galois structures arising from $P$ and $N/P$. Since $P$ is normalized by $N$, then, as discussed earlier $H_P=(K[P])^{\lG}$ is a $k$-sub-Hopf algebra of $H_N=(K[N])^{\lG}$ which fixes a subfield $F$, and that $F\otimes H_P$ acts to make $K/F$ Hopf-Galois, and by Thm. \ref{mainFknormal} $F/k$ is Hopf-Galois with respect to $H_{N/P}=(F[N/P])^{\lambda(G)/\lambda(J)}$. As such, we have the following.
\begin{theorem}
The following sequence of Hopf-algebras 
$$
1\rightarrow (K[P])^{\lambda(G)}\rightarrow (K[N])^{\lambda(G)}\rightarrow (F[N/P])^{\lambda(G)/\lambda(J)}\rightarrow 1
$$
is exact.
\end{theorem}
\begin{proof}
The reason this is exact is that we can rewrite the last term  $(F[N/P])^{\lambda(G)/\lambda(J)}$ as 
$$ 
\underbrace{(\underbrace{(K[N/P])^{\lambda(J)}}_{descend\ from\ K\ to\ F})^{\lambda(G)/\lambda(J)}}_{descend\ from\ F\ to\ k}=(K[N/P])^{\lambda(G)}
$$
since $(K[N/P])^{\lambda(J)}=F[N/P]$ due to $\lambda(J)$ acting trivially on $N/P$ as one can see from the proof of Thm. \ref{QnormQ}.
So in the second row below
$$\diagram
 1 \rto                 & K[P] \rto\dto_{\lambda(G)}              & K[N] \rto\dto_{\lambda(G)}              &  K[N/P] \rto\dto_{\lambda(G)}             & 1 \\
 1 \rto                 & (K[P])^{\lambda(G)} \rto                & (K[N])^{\lambda(G)} \rto                & (K[N/P])^{\lambda(G)} \rto        & 1 \\
\enddiagram$$
exactness is due to faithful flatness. That is, the upper row being exact implies that the lower row, consisting of the fixed rings, under the action of $\lambda(G)$, is exact.
\end{proof}
Note, unsurprisingly, that by Thm. \ref{mainFknon-normal} $(K[N/P])^{\lambda(G)}$ is exactly the Hopf algebra which acts on $F/k$ when $J$ is not normal in $G$. Of course, in that setting $G/J$ does not act on $F/k$ since each coset representative corresponds to a different conjugate field of $F$.
\section{Examples}
We conclude with a set of examples whose definition partly explains the notational choices for the sizes of the relevant groups given earlier. The condition that $P\triangleleft N$ is, in itself, not a rare one at all. However, the requirement that $P$ be normalized by $G$ as well is much more stringent.\par
One class of examples where this is automatic are those subgroups $P$ which are characteristic in $N$. Being characteristic is a feature of certain classes of subgroups, for example centers, Frattini subgroups, commutators etc.\par
Another more basic condition is uniqueness based on order considerations. Specifically suppose $|G|=|N|=mp$ where $p$ is prime, $gcd(p,m)=1$, and the Sylow $p$-subgroup is unique, such as when $p>m$. In this situation the Sylow $p$-subgroup is characteristic, and moreover $J=\Psi(P)$ must therefore be the similarly unique Sylow $p$-subgroup of $G$ as well. As such, $J\triangleleft G$ and therefore $F/k$ is Galois with respect to $G/J$ and Hopf-Galois with respect to $(F[N/P])^{\lambda(G)/\lambda(J)}$.\par
We include below tables for the 6 groups of order 42 mentioned above, $C_2\times (C_7\rtimes C_3)$, $C_3\times D_7$, $C_{42}$, $(C_7\rtimes C_3)\rtimes C_2$, $C_7\times D_3$, and $D_{21}$ where $p=7$ and $m=6$. In these tables below, we restrict to those $P$ normal in $N$. We denote by $[N]$ the isomorphism class of $N$, and similarly $[P]$ and $[J]$, and where the first column is the number of such triples $([N],[P],[J])$ for the given $G$, and the last column indicates whether $J$ was normal in $G$ or $|I|=[K:\tilde F]$ if not. One point to reiterate is that for each $P$, the group $J=\Psi(P)$ need not be isomorphic to $P$.\par
\newpage
$G \cong D_{{21}}$\par
\begin{tabular}{|c|c|c|c|c|c|}\hline
\#$N$ & $[N]$ & $[P]$  & $[J]$ & $J$ normal in $G$ or not \\ \hline
  6 & $C_{3} \times D_{7}$ & $C_{21}$ & $C_{21}$ & $ J \triangleleft G$ \\ \hline 
      6 & $C_{3} \times D_{7}$ & $C_{3}$ & $C_{3}$ & $ J \triangleleft G$ \\ \hline 
      6 & $C_{3} \times D_{7}$ & $C_{7}$ & $C_{7}$ & $ J \triangleleft G$ \\ \hline 
      6 & $C_{3} \times D_{7}$ & $D_{7}$ & $D_{7}$ & $|I|=7$ \\ \hline 
     21 & $C_{42}$ & $C_{14}$ & $D_{7}$ & $|I|=7$ \\ \hline 
     21 & $C_{42}$ & $C_{21}$ & $C_{21}$ & $ J \triangleleft G$ \\ \hline 
     21 & $C_{42}$ & $C_{2}$ & $C_{2}$ & $|I|=1$ \\ \hline 
     21 & $C_{42}$ & $C_{3}$ & $C_{3}$ & $ J \triangleleft G$ \\ \hline 
     21 & $C_{42}$ & $C_{6}$ & $D_{3}$ & $|I|=3$ \\ \hline 
     21 & $C_{42}$ & $C_{7}$ & $C_{7}$ & $ J \triangleleft G$ \\ \hline 
     14 & $C_{7} \times D_{3}$ & $C_{21}$ & $C_{21}$ & $ J \triangleleft G$ \\ \hline 
     14 & $C_{7} \times D_{3}$ & $C_{3}$ & $C_{3}$ & $ J \triangleleft G$ \\ \hline 
     14 & $C_{7} \times D_{3}$ & $C_{7}$ & $C_{7}$ & $ J \triangleleft G$ \\ \hline 
     14 & $C_{7} \times D_{3}$ & $D_{3}$ & $D_{3}$ & $|I|=3$ \\ \hline 
      4 & $D_{21}$ & $C_{21}$ & $C_{21}$ & $ J \triangleleft G$ \\ \hline 
      4 & $D_{21}$ & $C_{3}$ & $C_{3}$ & $ J \triangleleft G$ \\ \hline 
      4 & $D_{21}$ & $C_{7}$ & $C_{7}$ & $ J \triangleleft G$ \\ \hline 
\end{tabular}\par\newpage
$G\cong  C_{2} \times (C_{7} \rtimes C_{3}) $\par
\begin{tabular}{|c|c|c|c|c|c|}\hline
\#$N$ & $[N]$ & $[P]$  & $[J]$ & $J$ normal in $G$ or not \\ \hline
   16 & $C_{2} \times (C_{7} \rtimes C_{3})$ & $C_{14}$ & $C_{14}$ & $ J \triangleleft G$ \\ \hline 
     16 & $C_{2} \times (C_{7} \rtimes C_{3})$ & $C_{2}$ & $C_{2}$ & $ J \triangleleft G$ \\ \hline 
     16 & $C_{2} \times (C_{7} \rtimes C_{3})$ & $C_{7} \rtimes C_{3}$ & $C_{7} \rtimes C_{3}$ & $ J \triangleleft G$ \\ \hline 
     16 & $C_{2} \times (C_{7} \rtimes C_{3})$ & $C_{7}$ & $C_{7}$ & $ J \triangleleft G$ \\ \hline 
     14 & $C_{3} \times D_{7}$ & $C_{21}$ & $C_{7} \rtimes C_{3}$ & $ J \triangleleft G$ \\ \hline 
     14 & $C_{3} \times D_{7}$ & $C_{3}$ & $C_{3}$ & $|I|=1$ \\ \hline 
     14 & $C_{3} \times D_{7}$ & $C_{7}$ & $C_{7}$ & $ J \triangleleft G$ \\ \hline 
     14 & $C_{3} \times D_{7}$ & $D_{7}$ & $C_{14}$ & $ J \triangleleft G$ \\ \hline 
      7 & $C_{42}$ & $C_{14}$ & $C_{14}$ & $ J \triangleleft G$ \\ \hline 
      7 & $C_{42}$ & $C_{21}$ & $C_{7} \rtimes C_{3}$ & $ J \triangleleft G$ \\ \hline 
      7 & $C_{42}$ & $C_{2}$ & $C_{2}$ & $ J \triangleleft G$ \\ \hline 
      7 & $C_{42}$ & $C_{3}$ & $C_{3}$ & $|I|=1$ \\ \hline 
      7 & $C_{42}$ & $C_{6}$ & $C_{6}$ & $|I|=2$ \\ \hline 
      7 & $C_{42}$ & $C_{7}$ & $C_{7}$ & $ J \triangleleft G$ \\ \hline 
     28 & $(C_{7} \rtimes C_{3}) \rtimes C_{2}$ & $C_{7} \rtimes C_{3}$ & $C_{7} \rtimes C_{3}$ & $ J \triangleleft G$ \\ \hline 
     28 & $(C_{7} \rtimes C_{3}) \rtimes C_{2}$ & $C_{7}$ & $C_{7}$ & $ J \triangleleft G$ \\ \hline 
     28 & $(C_{7} \rtimes C_{3}) \rtimes C_{2}$ & $D_{7}$ & $C_{14}$ & $ J \triangleleft G$ \\ \hline 
     14 & $C_{7} \times D_{3}$ & $C_{21}$ & $C_{7} \rtimes C_{3}$ & $ J \triangleleft G$ \\ \hline 
     14 & $C_{7} \times D_{3}$ & $C_{3}$ & $C_{3}$ & $|I|=1$ \\ \hline 
     14 & $C_{7} \times D_{3}$ & $C_{7}$ & $C_{7}$ & $ J \triangleleft G$ \\ \hline 
     14 & $C_{7} \times D_{3}$ & $D_{3}$ & $C_{6}$ & $|I|=2$ \\ \hline 
     28 & $D_{21}$ & $C_{21}$ & $C_{7} \rtimes C_{3}$ & $ J \triangleleft G$ \\ \hline 
     28 & $D_{21}$ & $C_{3}$ & $C_{3}$ & $|I|=1$ \\ \hline 
     28 & $D_{21}$ & $C_{7}$ & $C_{7}$ & $ J \triangleleft G$ \\ \hline 
\end{tabular}\par\newpage
$G \cong C_{3} \times D_{7}$\par
\begin{tabular}{|c|c|c|c|c|c|}\hline
\#$N$ & $[N]$ & $[P]$  & $[J]$ & $J$ normal in $G$ or not \\ \hline
  28 & $C_{2} \times (C_{7} \rtimes C_{3})$ & $C_{14}$ & $D_{7}$ & $ J \triangleleft G$ \\ \hline 
     28 & $C_{2} \times (C_{7} \rtimes C_{3})$ & $C_{2}$ & $C_{2}$ & $|I|=1$ \\ \hline 
     28 & $C_{2} \times (C_{7} \rtimes C_{3})$ & $C_{7} \rtimes C_{3}$ & $C_{21}$ & $ J \triangleleft G$ \\ \hline 
     28 & $C_{2} \times (C_{7} \rtimes C_{3})$ & $C_{7}$ & $C_{7}$ & $ J \triangleleft G$ \\ \hline 
      2 & $C_{3} \times D_{7}$ & $C_{21}$ & $C_{21}$ & $ J \triangleleft G$ \\ \hline 
      2 & $C_{3} \times D_{7}$ & $C_{3}$ & $C_{3}$ & $ J \triangleleft G$ \\ \hline 
      2 & $C_{3} \times D_{7}$ & $C_{7}$ & $C_{7}$ & $ J \triangleleft G$ \\ \hline 
      2 & $C_{3} \times D_{7}$ & $D_{7}$ & $D_{7}$ & $ J \triangleleft G$ \\ \hline 
      7 & $C_{42}$ & $C_{14}$ & $D_{7}$ & $ J \triangleleft G$ \\ \hline 
      7 & $C_{42}$ & $C_{21}$ & $C_{21}$ & $ J \triangleleft G$ \\ \hline 
      7 & $C_{42}$ & $C_{2}$ & $C_{2}$ & $|I|=1$ \\ \hline 
      7 & $C_{42}$ & $C_{3}$ & $C_{3}$ & $ J \triangleleft G$ \\ \hline 
      7 & $C_{42}$ & $C_{6}$ & $C_{6}$ & $|I|=3$ \\ \hline 
      7 & $C_{42}$ & $C_{7}$ & $C_{7}$ & $ J \triangleleft G$ \\ \hline 
     28 & $(C_{7} \rtimes C_{3}) \rtimes C_{2}$ & $C_{7} \rtimes C_{3}$ & $C_{21}$ & $ J \triangleleft G$ \\ \hline 
     28 & $(C_{7} \rtimes C_{3}) \rtimes C_{2}$ & $C_{7}$ & $C_{7}$ & $ J \triangleleft G$ \\ \hline 
     28 & $(C_{7} \rtimes C_{3}) \rtimes C_{2}$ & $D_{7}$ & $D_{7}$ & $ J \triangleleft G$ \\ \hline 
     14 & $C_{7} \times D_{3}$ & $C_{21}$ & $C_{21}$ & $ J \triangleleft G$ \\ \hline 
     14 & $C_{7} \times D_{3}$ & $C_{3}$ & $C_{3}$ & $ J \triangleleft G$ \\ \hline 
     14 & $C_{7} \times D_{3}$ & $C_{7}$ & $C_{7}$ & $ J \triangleleft G$ \\ \hline 
     14 & $C_{7} \times D_{3}$ & $D_{3}$ & $C_{6}$ & $|I|=3$ \\ \hline 
      4 & $D_{21}$ & $C_{21}$ & $C_{21}$ & $ J \triangleleft G$ \\ \hline 
      4 & $D_{21}$ & $C_{3}$ & $C_{3}$ & $ J \triangleleft G$ \\ \hline 
      4 & $D_{21}$ & $C_{7}$ & $C_{7}$ & $ J \triangleleft G$ \\ \hline 
\end{tabular}\par\newpage
$G \cong C_{42}$\par
\begin{tabular}{|c|c|c|c|c|c|}\hline
\#$N$ & $[N]$ & $[P]$  & $[J]$ & $J$ normal in $G$ or not \\ \hline
     4 & $C_{2} \times (C_{7} \rtimes C_{3})$ & $C_{14}$ & $C_{14}$ & $ J \triangleleft G$ \\ \hline 
      4 & $C_{2} \times (C_{7} \rtimes C_{3})$ & $C_{2}$ & $C_{2}$ & $ J \triangleleft G$ \\ \hline 
      4 & $C_{2} \times (C_{7} \rtimes C_{3})$ & $C_{7} \rtimes C_{3}$ & $C_{21}$ & $ J \triangleleft G$ \\ \hline 
      4 & $C_{2} \times (C_{7} \rtimes C_{3})$ & $C_{7}$ & $C_{7}$ & $ J \triangleleft G$ \\ \hline 
      2 & $C_{3} \times D_{7}$ & $C_{21}$ & $C_{21}$ & $ J \triangleleft G$ \\ \hline 
      2 & $C_{3} \times D_{7}$ & $C_{3}$ & $C_{3}$ & $ J \triangleleft G$ \\ \hline 
      2 & $C_{3} \times D_{7}$ & $C_{7}$ & $C_{7}$ & $ J \triangleleft G$ \\ \hline 
      2 & $C_{3} \times D_{7}$ & $D_{7}$ & $C_{14}$ & $ J \triangleleft G$ \\ \hline 
      1 & $C_{42}$ & $C_{14}$ & $C_{14}$ & $ J \triangleleft G$ \\ \hline 
      1 & $C_{42}$ & $C_{21}$ & $C_{21}$ & $ J \triangleleft G$ \\ \hline 
      1 & $C_{42}$ & $C_{2}$ & $C_{2}$ & $ J \triangleleft G$ \\ \hline 
      1 & $C_{42}$ & $C_{3}$ & $C_{3}$ & $ J \triangleleft G$ \\ \hline 
      1 & $C_{42}$ & $C_{6}$ & $C_{6}$ & $ J \triangleleft G$ \\ \hline 
      1 & $C_{42}$ & $C_{7}$ & $C_{7}$ & $ J \triangleleft G$ \\ \hline 
      4 & $(C_{7} \rtimes C_{3}) \rtimes C_{2}$ & $C_{7} \rtimes C_{3}$ & $C_{21}$ & $ J \triangleleft G$ \\ \hline 
      4 & $(C_{7} \rtimes C_{3}) \rtimes C_{2}$ & $C_{7}$ & $C_{7}$ & $ J \triangleleft G$ \\ \hline 
      4 & $(C_{7} \rtimes C_{3}) \rtimes C_{2}$ & $D_{7}$ & $C_{14}$ & $ J \triangleleft G$ \\ \hline 
      2 & $C_{7} \times D_{3}$ & $C_{21}$ & $C_{21}$ & $ J \triangleleft G$ \\ \hline 
      2 & $C_{7} \times D_{3}$ & $C_{3}$ & $C_{3}$ & $ J \triangleleft G$ \\ \hline 
      2 & $C_{7} \times D_{3}$ & $C_{7}$ & $C_{7}$ & $ J \triangleleft G$ \\ \hline 
      2 & $C_{7} \times D_{3}$ & $D_{3}$ & $C_{6}$ & $ J \triangleleft G$ \\ \hline 
      4 & $D_{21}$ & $C_{21}$ & $C_{21}$ & $ J \triangleleft G$ \\ \hline 
      4 & $D_{21}$ & $C_{3}$ & $C_{3}$ & $ J \triangleleft G$ \\ \hline 
      4 & $D_{21}$ & $C_{7}$ & $C_{7}$ & $ J \triangleleft G$ \\ \hline 
\end{tabular}\par\newpage
$G \cong (C_{7} \rtimes C_{3}) \rtimes C_{2}$\par
\begin{tabular}{|c|c|c|c|c|c|}\hline
\#$N$ & $[N]$ & $[P]$  & $[J]$ & $J$ normal in $G$ or not \\ \hline
    28 & $C_{2} \times (C_{7} \rtimes C_{3})$ & $C_{14}$ & $D_{7}$ & $ J \triangleleft G$ \\ \hline 
     28 & $C_{2} \times (C_{7} \rtimes C_{3})$ & $C_{2}$ & $C_{2}$ & $|I|=1$ \\ \hline 
     28 & $C_{2} \times (C_{7} \rtimes C_{3})$ & $C_{7} \rtimes C_{3}$ & $C_{7} \rtimes C_{3}$ & $ J \triangleleft G$ \\ \hline 
     28 & $C_{2} \times (C_{7} \rtimes C_{3})$ & $C_{7}$ & $C_{7}$ & $ J \triangleleft G$ \\ \hline 
     14 & $C_{3} \times D_{7}$ & $C_{21}$ & $C_{7} \rtimes C_{3}$ & $ J \triangleleft G$ \\ \hline 
     14 & $C_{3} \times D_{7}$ & $C_{3}$ & $C_{3}$ & $|I|=1$ \\ \hline 
     14 & $C_{3} \times D_{7}$ & $C_{7}$ & $C_{7}$ & $ J \triangleleft G$ \\ \hline 
     14 & $C_{3} \times D_{7}$ & $D_{7}$ & $D_{7}$ & $ J \triangleleft G$ \\ \hline 
      7 & $C_{42}$ & $C_{14}$ & $D_{7}$ & $ J \triangleleft G$ \\ \hline 
      7 & $C_{42}$ & $C_{21}$ & $C_{7} \rtimes C_{3}$ & $ J \triangleleft G$ \\ \hline 
      7 & $C_{42}$ & $C_{2}$ & $C_{2}$ & $|I|=1$ \\ \hline 
      7 & $C_{42}$ & $C_{3}$ & $C_{3}$ & $|I|=1$ \\ \hline 
      7 & $C_{42}$ & $C_{6}$ & $C_{6}$ & $|I|=1$ \\ \hline 
      7 & $C_{42}$ & $C_{7}$ & $C_{7}$ & $ J \triangleleft G$ \\ \hline 
     16 & $(C_{7} \rtimes C_{3}) \rtimes C_{2}$ & $C_{7} \rtimes C_{3}$ & $C_{7} \rtimes C_{3}$ & $ J \triangleleft G$ \\ \hline 
     16 & $(C_{7} \rtimes C_{3}) \rtimes C_{2}$ & $C_{7}$ & $C_{7}$ & $ J \triangleleft G$ \\ \hline 
     16 & $(C_{7} \rtimes C_{3}) \rtimes C_{2}$ & $D_{7}$ & $D_{7}$ & $ J \triangleleft G$ \\ \hline 
     14 & $C_{7} \times D_{3}$ & $C_{21}$ & $C_{7} \rtimes C_{3}$ & $ J \triangleleft G$ \\ \hline 
     14 & $C_{7} \times D_{3}$ & $C_{3}$ & $C_{3}$ & $|I|=1$ \\ \hline 
     14 & $C_{7} \times D_{3}$ & $C_{7}$ & $C_{7}$ & $ J \triangleleft G$ \\ \hline 
     14 & $C_{7} \times D_{3}$ & $D_{3}$ & $C_{6}$ & $|I|=1$ \\ \hline 
     28 & $D_{21}$ & $C_{21}$ & $C_{7} \rtimes C_{3}$ & $ J \triangleleft G$ \\ \hline 
     28 & $D_{21}$ & $C_{3}$ & $C_{3}$ & $|I|=1$ \\ \hline 
     28 & $D_{21}$ & $C_{7}$ & $C_{7}$ & $ J \triangleleft G$ \\ \hline 
\end{tabular}\par\newpage
$G \cong C_{7} \times D_3$\par
\begin{tabular}{|c|c|c|c|c|c|}\hline
\#$N$ & $[N]$ & $[P]$  & $[J]$ & $J$ normal in $G$ or not \\ \hline
    6 & $C_{3} \times D_{7}$ & $C_{21}$ & $C_{21}$ & $ J \triangleleft G$ \\ \hline 
      6 & $C_{3} \times D_{7}$ & $C_{3}$ & $C_{3}$ & $ J \triangleleft G$ \\ \hline 
      6 & $C_{3} \times D_{7}$ & $C_{7}$ & $C_{7}$ & $ J \triangleleft G$ \\ \hline 
      6 & $C_{3} \times D_{7}$ & $D_{7}$ & $C_{14}$ & $|I|=7$ \\ \hline 
      3 & $C_{42}$ & $C_{14}$ & $C_{14}$ & $|I|=7$ \\ \hline 
      3 & $C_{42}$ & $C_{21}$ & $C_{21}$ & $ J \triangleleft G$ \\ \hline 
      3 & $C_{42}$ & $C_{2}$ & $C_{2}$ & $|I|=1$ \\ \hline 
      3 & $C_{42}$ & $C_{3}$ & $C_{3}$ & $ J \triangleleft G$ \\ \hline 
      3 & $C_{42}$ & $C_{6}$ & $D_{3}$ & $ J \triangleleft G$ \\ \hline 
      3 & $C_{42}$ & $C_{7}$ & $C_{7}$ & $ J \triangleleft G$ \\ \hline 
      2 & $C_{7} \times D_{3}$ & $C_{21}$ & $C_{21}$ & $ J \triangleleft G$ \\ \hline 
      2 & $C_{7} \times D_{3}$ & $C_{3}$ & $C_{3}$ & $ J \triangleleft G$ \\ \hline 
      2 & $C_{7} \times D_{3}$ & $C_{7}$ & $C_{7}$ & $ J \triangleleft G$ \\ \hline 
      2 & $C_{7} \times D_{3}$ & $D_{3}$ & $D_{3}$ & $ J \triangleleft G$ \\ \hline 
      4 & $D_{21}$ & $C_{21}$ & $C_{21}$ & $ J \triangleleft G$ \\ \hline 
      4 & $D_{21}$ & $C_{3}$ & $C_{3}$ & $ J \triangleleft G$ \\ \hline 
      4 & $D_{21}$ & $C_{7}$ & $C_{7}$ & $ J \triangleleft G$ \\ \hline 
\end{tabular}\par

\bibliography{structure}
\bibliographystyle{plain}
\end{document}